\documentclass[]{interact}

\usepackage[numbers,sort&compress]{natbib}
\bibpunct[, ]{[}{]}{,}{n}{,}{,}
\renewcommand\bibfont{\fontsize{10}{12}\selectfont}

\usepackage{caption}
\usepackage{algorithmic}
\usepackage[ruled,linesnumbered,vlined]{algorithm2e}
\usepackage{hyperref}
\usepackage{multirow}
\usepackage{subcaption}
\newdimen\figrasterwd
\figrasterwd\textwidth
\usepackage[shortlabels]{enumitem}
\usepackage{xcolor}

\theoremstyle{plain}
\newtheorem{theorem}{Theorem}[section]
\newtheorem{lemma}[theorem]{Lemma}

\newtheorem{proposition}[theorem]{Proposition}

\theoremstyle{definition}
\newtheorem{definition}[theorem]{Definition}

\newtheorem{assumption}{Assumption}

\theoremstyle{remark}
\newtheorem{remark}{Remark}

\usepackage{tikz}
\usetikzlibrary{calc}
\usepackage{pgfplots}
\pgfplotsset{compat=newest}
\usetikzlibrary{arrows,arrows.meta,intersections,shapes.geometric,positioning,calc,patterns}

\newcommand{\rev}[1]{{#1}}
\begin{document}

\articletype{ARTICLE}

\title{Projection-based curve pattern search for black-box optimization over smooth convex sets}

\author{
\name{Xiaoxi Jia\textsuperscript{a}, Matteo Lapucci\textsuperscript{b}\thanks{CONTACT Matteo Lapucci. Email: matteo.lapucci@unifi.it}, Pierluigi Mansueto\textsuperscript{b}}
\affil{\textsuperscript{a}Independent researcher\\ \textsuperscript{b}Dipartimento di Ingegneria dell'Informazione, Università di Firenze, Via di S.\ Marta, 3, 50139, Firenze, Italy}
}

\renewcommand{\thefigure}{\roman{figure}}
\renewcommand{\thetable}{\roman{table}}
\renewcommand{\thepage}{\roman{page}}
\renewcommand{\thefigure}{\arabic{figure}}
\renewcommand{\thetable}{\arabic{table}}
\renewcommand{\thepage}{\arabic{page}}
\setcounter{figure}{0}
\setcounter{table}{0}
\setcounter{page}{0}

\maketitle

\begin{abstract}
In this paper, we deal with the problem  of optimizing a black-box smooth function over a full-dimensional smooth convex set. We study sets of feasible curves that allow to properly characterize stationarity of a solution and possibly carry out sound backtracking curvilinear searches. We then propose a general pattern search algorithmic framework that exploits curves of this type to carry out poll steps and for which we prove properties of asymptotic convergence to stationary points. We particularly point out that the proposed framework covers the case where search curves are arcs induced by the Euclidean projection of coordinate directions. The method is finally proved to arguably be superior, on smooth problems, than other recent projection-based algorithms \rev{and is competitive with state-of-the-art methods} from the literature on constrained black-box optimization.
\end{abstract}


\begin{amscode}
	90C56, 90C30, 90C26 
\end{amscode}

\section{Introduction}
\label{sec::intro}

In this manuscript we consider the problem of minimizing a continuously differentiable function $f:C \to\mathbb{R}$ within a compact convex feasible set $C\subset \mathbb R^n$, i.e.,
\begin{equation}
\label{Eq:P} 
\begin{aligned}
    \min \ & f(x)\\  \text{s.t. }& x\in C.
\end{aligned}
\end{equation}
More specifically, we assume that $C$ is an $n$-dimensional compact object with a smooth boundary: in other words, there exists some continuously differentiable function $g:\mathbb{R}^n\to\mathbb{R}$ such that $C=\{x\in\mathbb{R}^n\mid g(x)\le 0\}$, its frontier is given by $\partial C=\{x\in\mathbb{R}^n\mid g(x)= 0\}$ and there exists $x\in C$ such that $g(x)<0$, i.e., $x\in\operatorname{int}C$. While function $g$ might be unknown, we assume to have access to the Euclidean projection operator onto this set $P_C:\mathbb{R}^n\to C$ and that the cost of computing projection is sustainable.

On the other hand, while $f$ is smooth, we assume not to have direct access to its derivatives: we are thus in a setting of black-box optimization. We assume that $f$ is bounded below on $C$ by some value $f^*$ and that the value of $f$ cannot be computed outside the feasible set. 

A vast literature has dealt with black-box optimization problems; we refer the reader to \cite{larson2019derivative} for a recent survey on derivative-free optimization. In this paper, we put particular focus on the class of pattern search methods \cite{hooke1961direct,lewis1999pattern,lewis2000pattern,audet2002analysis,custodio2007using,ortho09} and specifically on pattern search algorithms based on line searches. These methods, whose study has roots in a couple of pioneering works by Grippo et al.\  \cite{de1984stopping,grippo1988global}, includes effective algorithms to tackle constrained \cite{Lucidi2002Derivative,lucidi2002objective} and unconstrained \cite{lucidi2002global} black-box optimization problems, even in absence of smoothness assumptions on the objective function \cite{fasano2014linesearch}.

For the particular case of constrained problems with smooth black-box objective, which is the focus of the present work, the classical approach based on line-searches along coordinate directions (or other suitable, predefined sets of directions) has naturally be extended to deal with bound constraints \cite{Lucidi2002Derivative}; with more complex constraints, the set of search directions shall take into account the local structure of the feasible set: while this can be reasonably done with general linear constraints, when constraints are nonlinear much more caution is required \cite{lucidi2002objective}. 

A different path to handling constraints within pattern search frameworks is based on penalty approaches, either in a sequential \cite{liuzzi2010sequential,lewis2002globally} or exact fashion \cite{di2015derivative}:  the complexity of handling constraints is moved to the objective function, which can then be optimized with methods for unconstrained problems. While effective in various settings, this latter class approaches suffer from the issue of possibly needing objective function evaluations even outside the feasible set, where the black-box might be not well defined. Moreover, this kind of methods often suffer from the high sensitivity to the choice of the penalty parameters, that heavily impact the performance of the algorithm. 

Recently, a strategy was proposed in \cite{galvan2021parameter} that overcomes these limitations under the assumption that the projection operation onto the feasible set is available: problem \eqref{Eq:P} is equivalently reformulated as an unconstrained nonsmooth problem where the objective function at any point $x$ is equal to the sum of the original objective computed at the projection of $x$ onto $C$ and the distance of $x$ from the set $C$ itself; clearly, the value corresponds to the value of the original objective when $x\in C$. The equivalent unconstrained problem can then be tackled with the CS-DFN method from \cite{fasano2014linesearch} for nonsmooth problems. 

While the projection based approach from \cite{galvan2021parameter} has some connections to the fundamental ideas of projected gradient methods for first-order optimization \cite[Sec.\ 2.3]{bertsekas1999nonlinear}, it cannot be actually seen as a zeroth-order adaptation of the latter. In fact, while points obtained by steps along given directions are projected onto the feasible sets in both frameworks, the projected gradient method produces a sequence of iterates all belonging to the feasible set, whereas the approach from \cite{galvan2021parameter} possibly moves along unfeasible solutions; when this happens, it may come with some practical drawbacks: 
first, the points to be polled will be (almost) all unfeasible, thus requiring a high number of projections; this might result in a significant cost if projection onto $C$ is not particularly cheap. Moreover, not only the presence of a penalty part within the auxiliary objective might end up slowing down the optimization process, but its nonsmoothness also makes convergence dependent on the employment of random search directions. 

The main contribution of this paper consists in a more consistent adaptation of the projected gradient method to the black-box case, under the assumption that $C$ is a full-dimensional smooth convex subset of the Euclidean space. Such an extension, which overcomes the aforementioned drawbacks of the approach from \cite{galvan2021parameter}, requires to study the behavior of searches along suitable projection arcs induced by a predefined set of search directions.
To the best of our knowledge, the idea of a curve pattern search that polls the objective at points obtained, for example, by the projection of coordinate steps is novel in the literature.

\rev{\label{rev:intro} It may be important to underline here that we have a different goal compared to other works, such as  the recent contribution \cite{Custódio03052024}. In the latter paper, authors suggest to emulate actual projected-gradient steps in the black-box scenario: this can be done employing a numerically approximated gradient smartly obtained using the information provided by sampled points. We rather aim to insert the projection mechanism directly in the poll step: such a technique might induce benefits like, e.g., guaranteed feasibility of sampled points. 
In fact, a similar strategy as the one we propose in this manuscript is mentioned (by the name DDS-proj) as a heuristic and tested in the experimental section of \cite{Custódio03052024}, either as a standalone method and as a way to improve the polling step for the simplex-gradient type approach; however, in contrast to this work, the theoretical properties of DDS-proj are not formally discussed there. 

More in general, for our method there will not even be the need for a dense of search directions to achieve convergence solely via polling in constrained problems. This property does not hold true, even when $f$ is smooth, for other state-of-the-art pattern-search methods in the literature like, e.g., \cite{ortho09} and \cite{Custódio03052024} itself. 

All in all, despite a significantly simplified structure, the algorithm we propose to use in our experiments possesses convergence guarantees and shows competitive performance w.r.t.\ state-of-the-art alternatives.

}

The rest of the paper is organized as follows: in Section \ref{sec::feasible_set}, we review known results concerning the characterization of the feasible sets considered in this manuscript and the corresponding stationary solutions; then, we turn in Section \ref{sec:curves} to the definition of the concept of feasible search path and the interconnection between stationarity of solutions and suitable sets of search curves. We present the novel algorithmic framework in Section \ref{sec:alg}, providing the characterization of a gradient-based version in Section \ref{sec:alg_1o} and then presenting and theoretically analyzing the derivative-free method in Section \ref{sec:alg_df}. In Section \ref{sec::num_res} we present the results of computational experiments assessing efficiency and effectiveness of the proposed method - implemented exploiting coordinate directions and projection operation - compared to another pattern search approach \rev{and a selection of state-of-the-art methods} from the literature. We finally give concluding remarks in Section \ref{sec:concl}.

\section{Characterization of the feasible set}
\label{sec::feasible_set}

As stated in the introduction, in this work we deal with feasible sets $C$ that satisfy the following assumptions:
\begin{enumerate}
    \item[(a1)] $C$ is convex;
    \item[(a2)] $C$ is compact;
    \item[(a3)] there exists $g:\mathbb{R}^n\to \mathbb{R}$ continuously differentiable such that $C=\{x\in\mathbb{R}^n\mid g(x)\le 0\}$ and $g(x)<0$ for some $x\in C$.
\end{enumerate}

As with any convex feasible set, we can characterize points $x$ in the feasible region by the corresponding \textit{tangent cone} $T_C(x)$, defined as \cite[Def.\ 3.3.1]{bertsekas1999nonlinear}:
$$T_C(x) = \{0\}\cup\left\{y\in\mathbb{R}^n\,\middle\vert\, \exists\{x^k\}\subseteq C\setminus\{x\} \text{ s.t. }x^k\to x\text{ and }\frac{x^k-x}{\|x^k-x\|}\to \frac{y}{\|y\|}\right\}.$$
The tangent cone is always a convex cone if $C$ is convex. Moreover, $T_C(x)$ always contains all feasible directions at $x$, so it is easy to see that the stationarity property
\begin{equation}
    \label{eq:stat_tangent_cone}
    \nabla f(x)^Td\ge 0 \quad \forall\; d\in T_C(x)
\end{equation}
is a necessary condition of optimality for problem \eqref{Eq:P}.
This condition in general is not immediately checkable, since we are asking for the property to hold for infinitely many directions.

However, the condition becomes numerically easy to assess if we are able to do projections onto the tangent cone and we have direct access to the gradient of the objective function.
\begin{proposition}
    A point $\bar{x}\in C$ is a stationary point in the sense of \eqref{eq:stat_tangent_cone} for problem \eqref{Eq:P} if and only if we have
    $$P_{T_C(\bar{x})}(-\nabla f(\bar{x})) = 0.$$ 
\end{proposition}
\begin{proof}
    Assume $\bar{x}$ is stationary and let $\bar{d} = P_{T_C(\bar{x})}(-\nabla f(\bar{x}))$. Since $0\in T_C(\bar{x})$, by the properties of the projection onto the convex set $T_C(\bar{x})$ we have
    $$(-\nabla f(\bar{x})-\bar{d})^T(0-\bar{d})\le 0,$$
    i.e.,
    $$\nabla f(\bar{x})^T\bar{d}\le -\|\bar{d}\|^2.$$
    Since $\bar{d}\in T_C(\bar{x})$ by definition and $\bar{x}$ is stationary, we necessarily have $\|\bar{d}\|=0$.

    On the other hand, assume $\bar{d} = P_{T_C(\bar{x})}(-\nabla f(\bar{x}))=0$. Let $d$ be any direction in $T_C(\bar{x})$. By the properties of projection we can write this time
    $$(-\nabla f(\bar{x})-\bar{d})^T(d-\bar{d})\le 0,$$
    i.e., 
    $$0\ge (-\nabla f(\bar{x})-0)^T(d-0) = -\nabla f(\bar{x})^Td,$$ which finally implies $\nabla f(\bar{x})^Td\ge 0$. Since $d$ is arbitrary in $T_C(\bar{x})$, we get the thesis.
\end{proof}

Under assumption (a3), we can go even further.
Since Slater's \rev{C}onstraint \rev{Q}ualification \rev{(CQ)} is clearly verified, for every point $x\in C$ the tangent cone coincides with the cone of first order variations \rev{\label{rev2:V}$V_C(x)$} \cite[Sec.\ 3.3.6]{bertsekas1999nonlinear}:
$$T_C(x) = V_C(x) = \begin{cases}
    \{y\in\mathbb{R}^n\mid \nabla g(x)^Ty\le 0\}&\text{if } g(x)=0 \text{ (i.e., } x\in\partial C \text{)},\\
    \mathbb{R}^n&\text{otherwise},
\end{cases}$$
where $\nabla g(x)\neq 0$ for all $x\in\partial C$ (Slater's CQ implies \rev{\label{rev2::CQ}Mangasarian–Fromovitz} CQ, see again \cite[Sec.\ 3.3.6]{bertsekas1999nonlinear}). In other words, the tangent cone is the full space for points in the interior of the feasible set, whereas it is a halfspace that varies continuously along the frontier. 
As noted in \cite[Sec.\ 3]{lucidi2002objective}, when $T_C(x)$ is polyhedral\rev{,} as in our case, we can characterize a stationary point of problem \eqref{Eq:P} by a finite number of directions. In order to do so, we first need to recall the concept of positive span of a set of directions $D=\{d_1,\ldots,d_m\}$: we denote by $\operatorname{cone}(D)=\{v\in\mathbb{R}^n\mid \exists\,\beta\in\mathbb{R}_+^m \text{ s.t. } v=\beta_1d_1+\ldots+\beta_md_m\}$ the cone positively spanned by the directions in $D$. We are now able to state the following proposition \cite[Prop.\ 2]{lucidi2002objective}.
\begin{proposition}
\label{prop:tseng}
    Let $C$ be a convex set, $\bar{x}\in C$ and assume $T_C(\bar{x})$ is polyhedral. Then, $\bar{x}$ is a stationary point in the sense of \eqref{eq:stat_tangent_cone} for problem \eqref{Eq:P} if and only if
    $$\nabla f(\bar x)^Td\ge 0 \quad \forall\; d\in D,$$
    where $\operatorname{cone}(D)=T_C(\bar{x})$.
\end{proposition}

The above proposition will be crucial to tackle problems of the form \eqref{Eq:P} without having access to $\nabla f$.

\section{Searching along curves}
\label{sec:curves}
Since we will have to deal with problem \eqref{Eq:P} without the opportunity of following the gradient direction, we are interested in devising predefined patterns to explore for reducing the objective value from a given point. In the unconstrained case, any set $D$ of directions such that $\operatorname{cone}(D)=\mathbb{R}^n$ can be used as a basis for derivative-free line searches \cite{lucidi2002global}; the same goes for the fortunate case of bound constraints \cite{Lucidi2002Derivative}. In presence of more complex constraints the situation gets harder, and the set of search directions to follow has to be carefully identified for each encountered solution \cite{lucidi2002objective}. 
The goal of this work is to identify a predefined search scheme that does not require to study and build a set of search directions dependent on the particular current point before polling new solutions.

To this aim, we will consider curvilinear search paths, formalized according to the following definition.
\begin{definition}
    We say that $\gamma:\mathbb{R}_+\to C$ is a \textit{feasible search path} at $x\in C$ if $\gamma$ is a continuous curve such that $\gamma(0)=x$, $\gamma(t)\in C$ for all $t\in\mathbb{R}_+$ and $\gamma$ is differentiable in $t=0$. 
\end{definition}
By definition, every feasible search path $\gamma$ is associated with an \textit{initial velocity} $v=\gamma'(0)$. 
In the following, we will be interested in sets $\Gamma(x)$ of feasible descent paths, possibly dependent on the particular point $x\in C$, such that:
\begin{itemize}
    \item if $x$ is not stationary, a decrease of the objective function is attained by $\gamma(t)$ for at least one $\gamma \in \Gamma(x)$ and for values of $t$ sufficiently small;
    \item if there exists $\epsilon>0$ such that $f(\gamma (t))\ge f(\gamma(0))=f(x)$ for all $t\in(0,\epsilon)$ and $\gamma\in \Gamma(x)$, then $x$ is necessarily a stationary point.
\end{itemize}

The following two propositions ensure us that a set of curves $\Gamma(x)$ does the job if the velocities of the curves positively span the tangent cone at a point, i.e., if $\operatorname{cone}(D)=T_C(x)$ where $D=\{v=\gamma'(0)\mid\gamma\in \Gamma(x)\}$. \rev{\label{rev2::D}Here and in the following, we will denote by $\mathcal{D}_{f\circ \gamma}$ the derivative with respect to $t$ of the composite function $f(\gamma(t))$.}

\begin{proposition}\label{Prop:DecObj}
    Let $C$ be a convex set and let $\bar{x}\in C$ be a nonstationary point in the sense of \eqref{eq:stat_tangent_cone} (i.e., $\exists\,d\in T_C(\bar{x}):\nabla f(\bar{x})^Td<0$) such that $T_C(\bar{x})$ is polyhedral. Let $\Gamma(\bar{x})$ be a set of feasible search paths such that $\operatorname{cone}(D)= T_C(\bar{x})$ with $D=\{v=\gamma'(0)\mid\gamma\in \Gamma(\bar{x})\}$.
    Then, for any $\sigma>0$, there exists $\gamma\in\Gamma(\bar{x})$ and $\bar{t}>0$ such that, for all $t\in(0,\bar{t}]$, we have $f(\gamma(t))-f(\gamma(0))<-\sigma t^2<0$. 
\end{proposition}
\begin{proof}
    By Proposition \ref{prop:tseng}, there exists ${v}\in D$ such that $\nabla f(\bar{x})^T{v}<0$. Let $\gamma\in\Gamma(\bar{x})$ such that $v=\gamma'(0)$.

    Now, let us assume by contradiction that 
    $$f(\gamma(t_k))\ge f(\bar{x})-\sigma t_k^2$$
    for a sequence $\{t_k\}\subseteq\mathbb{R}_+$ such that $t_k\to 0^+$. Rearranging and dividing both sides by $t_k$, we get
    $$\frac{f(\gamma(t_k))-f(\gamma(0))}{t_k}\ge -\sigma t_k.$$
    Taking the limits for $k\to\infty$ \rev{\label{rev1::proof}and recalling the definition of the sequence $\{t_k\}$, we get $$\frac{f(\gamma(t_k))-f(\gamma(0))}{t_k}\to \mathcal{D}_{f\circ\gamma}(0), \qquad -\sigma t_k \to 0,$$ thus,}
    $$\mathcal{D}_{f\circ\gamma}(0)\ge 0.$$
    We can then note that, by the chain rule,
    $$\mathcal{D}_{f\circ\gamma}(0) = \nabla f(\gamma(0))^T\gamma'(0) = \nabla f(\bar{x})^Tv.$$
    We therefore get that, 
    $$\nabla f(\bar{x})^Tv\ge 0,$$
    which finally gives us a contradiction.
\end{proof}

\begin{proposition}
\label{prop:suff_cond_stat}
    Let $C$ be a convex set and let $\bar{x}\in C$  such that $T_C(\bar{x})$ is polyhedral. Let $\Gamma(\bar{x})$ be a set of feasible search paths such that $\text{cone}(D)= T_C(\bar{x})$ with $D=\{v=\gamma'(0)\mid\gamma\in \Gamma(\bar{x})\}$. Further assume that there exists $\epsilon>0$ such that $f(\gamma (t))\ge f(\gamma(0))=f(\bar{x})$ for all $t\in(0,\epsilon)$ and $\gamma\in \Gamma(\bar{x})$. Then, $\bar{x}$ is a stationary point for problem \eqref{Eq:P} in the sense of \eqref{eq:stat_tangent_cone}.
\end{proposition}
\begin{proof}
    By the assumptions, we know that for all $\gamma\in\Gamma(\bar{x})$ and for $t>0$ sufficiently small we have 
    $f(\gamma(t))-f(\gamma(0))\ge 0$ and thus 
    $$\frac{f(\gamma(t))-f(\gamma(0))}{t}\ge 0.$$
    Taking the limits for $t\to0^+$, we get
    \begin{equation}
        \label{eq:p1}
        \mathcal{D}_{f\circ\gamma}(0) = \nabla f(\gamma(0))^T\gamma'(0) = \nabla f(\bar{x})^Tv\ge 0 
    \end{equation}
    for all $v\in D$. Since by assumption $\operatorname{cone}(D) =T_C(\bar{x})$, we get the thesis from Proposition \ref{prop:tseng}.
\end{proof}

Proposition \ref{Prop:DecObj} actually provides us not only with a guaranteed decrease result at a nonstationary point, but also ensures us that we can attain a sufficient decrease.

However, the results in Propositions \ref{Prop:DecObj} and \ref{prop:suff_cond_stat} are clearly based on choices of $\Gamma(x)$ heavily dependent on the particular point $x$.

In the following, we are going to present a scheme, only based on the projection operator, implicitly defining curves that satisfy these suitable conditions at all points of a set $C$ satisfying assumptions (a1)-(a3).
Specifically, we are going to consider, at any point $x\in C$, the paths defined by curves of the form 
\begin{equation}
    \label{eq:curves_form}
    \gamma_y(t)=P_C(x+ty),\qquad y\in\mathcal{Y},
\end{equation}
where $\mathcal{Y}$ is a suitable set of directions (not necessarily feasible nor in the tangent cone at $x$).
We now start providing a preliminary result that ensures that curves of the form \eqref{eq:curves_form} actually define feasible search paths at $x$ with identifiable initial velocities.
\begin{lemma}
\label{prop:proj_makes_curves}
    Let $C$ be a convex set and $\bar{x}\in C$. Let $y\in\mathbb{R}^n$ be any direction. The curve $\gamma_y(t)=P_C(\bar{x}+ty)$ is a feasible search paths at $\bar{x}$, with an initial velocity $\gamma_y'(0)=P_{T_C(\bar{x})}(y)$. 
\end{lemma}
\begin{proof}
    The continuity of $\gamma_y$ straightforwardly follows from the continuity of the projection onto a convex set. Moreover, by the definition of projection, we trivially have $\gamma_y(t)\in C$ for all $t$ and $\gamma_y(0) = \bar{x}$.
    As for the differentiability in $t=0$, we know (see, e.g., \cite[eq.\ (2)]{Shapiro2016}) that $P_C$ is directionally differentiable at every feasible point $x\in C$, with $\mathcal{D}_{P_C}(x;d) = P_{T_C(x)}(d)$ for all $d\in\mathbb{R}^n$; we therefore have
    \begin{align*}
        P_{T_C(\bar{x})}(y)=\mathcal{D}_{P_C}(\bar{x};y)&= \lim_{t\to 0^+}\frac{P_C(\bar{x}+ty)-P_C(\bar{x})}{t}\\ &= \lim_{t\to 0^+}\frac{\gamma_y(t)-\gamma_y(0)}{t}=\gamma_y'(0).
    \end{align*}
\end{proof}

At this point, what we would like to identify is a (constant) set of directions $\mathcal{Y}\subseteq\mathbb{R}^n$ such that we are guaranteed that, at any point $x\in C$, we have $$\operatorname{cone}(\{P_{T_C(x)}(y)\mid y\in \mathcal{Y}\})=T_C(x).$$
As we are going to prove shortly, the set of coordinate directions $B=\{e_1,\ldots,e_n, \allowbreak -e_1,\ldots,-e_n\} = \{b_1,\ldots,b_{2n}\}$, with $e_i$ being the $i$-th element of the canonical basis, enjoys this property when $C$ satisfies assumptions (a1)-(a3). 
\begin{proposition}
\label{prop:veloc_canon}
    Let $C$ be a set satisfying assumptions (a1)-(a3) and let $x\in C$. Let  $D=\{v_i=P_{T_C(x)}(b_i)\mid b_i\in B\}$. Then $\operatorname{cone}(D)=T_C(x)$.
\end{proposition}
\begin{proof}
    There are two possible cases: $x\in\operatorname{int} C$ (\rev{$g(x)< 0$}) or $x\in\partial C$ ($g(x)=0$). We deal with the two cases separately.
    \begin{itemize}
        \item Let us assume $x\in\operatorname{int} C$. In this case $T_C(x)= \mathbb{R}^n$, so that $P_{T_C(x)}(b_i) = b_i$ for any $b_i\in B$; then we have $D=B$ and, therefore, $\operatorname{cone}(D) = \operatorname{cone}(B)$; since $B$ is trivially a positive spanning set of $\mathbb{R}^n$ and $T_C(x)=\mathbb{R}^n$, the result is proven for this case.
        \item Let us assume now that $x\in\partial C$.  Then we have $T_C(x) = \{d\in\mathbb{R}^n\mid \nabla g(x)^Td\le 0\}$ with $\nabla g(x)\neq 0$ \rev{(see Figure \ref{fig:proof} for some visual intuition)}.
        We first show that if $\bar{d}\notin T_C(x)$ then $\bar{d}\notin \operatorname{cone}(D)$, i.e., $\operatorname{cone}(D)\subseteq T_C(x)$. Since $\bar{d}\notin T_C(x)$, we have $\nabla g(x)^T\bar{d}> 0$. On the other hand, we know that $D=\{P_{T_C(x)}(b)\mid b\in B\}$, hence by the definition of $P_{T_C(x)}$ we have $\nabla g(x)^Tv\le 0$ for all $v\in D$. Let us assume by contradiction that $\beta_1,\ldots,\beta_{2n}\ge0$ exist such that $\bar{d} = \sum_{v_i\in D} \beta_iv_i$. We then have
        $$\nabla g(x)^T\bar{d} = \sum_{v_i\in D} \beta_i\nabla g(x)^Tv_i\le 0,$$
        which is absurd. Hence $\operatorname{cone}(D)\subseteq T_C(x)$. 

        We now show that $T_C(x)\subseteq \operatorname{cone}(D)$. Let $d$ be any direction in $T_C(x)$. We know that $\nabla g(x)^Td\le 0$. 
        Since $\nabla g(x)\neq 0$ and $\operatorname{cone}(B)=\mathbb{R}^n$, we are guaranteed that there exists $b_i\in B$ such that $\nabla g(x)^Tb_i< 0$ (and thus $b_i\in T_C(x)$ and $v_i=P_{T_C(x)}(b_i)=b_i$). 
        We then have, for all $t$,
        \begin{align*}
            \nabla g(x)^Td &= \nabla g(x)^T(d-tb_i+tb_i)\\
            & = \nabla g(x)^T(tb_i)+\nabla g(x)^T(d-tb_i).
        \end{align*}
        If we set $t^*=\frac{\nabla g(x)^Td}{\nabla g(x)^Tb_i}\ge0$,
        we can write
        $$d=p+t^*b_i = p+t^*v_i,$$
        with $p=d-t^*b_i$ $v_i\in D$ and $t^*\ge0$. Moreover, $p$ is such that $\nabla g(x)^Tp =\nabla g(x)^Td- \frac{\nabla g(x)^Td}{\nabla g(x)^Tb_i}\nabla g(x)^Tb_i=0$, hence $p\in H_{C}(x) = \{y\in\mathbb{R}^n\mid \nabla g(x)^Ty=0\}= \partial T_C(x)$.
        Then, $d\in \operatorname{cone}(D)$ if $p\in \operatorname{cone}(D)$.
        Now, let $\tilde{p}(t) = p+t\nabla g(x)$; since $\tilde{p}(t)^T\nabla g(x) = \nabla g(x)^Tp+t\|\nabla g(x)\|^2>0$ for all $t>0$, we have that $P_{T_C(x)}(\tilde{p}(t)) = P_{H_C(x)}(\tilde{p}(t))$.
        We shall now write
        $$\nabla g(x) = \sum_{i:\nabla_ig(x)>0}\nabla_ig(x)e_i+\sum_{i:\nabla_i g(x)<0}|\nabla_ig(x)|(-e_i).$$
        Moreover, there certainly exist $\lambda_1,\ldots,\lambda_n$ such that
        \begin{align*}
            p &= \sum_{i=1}^{n}\lambda_i e_i \\&= \sum_{i:\nabla_ig(x)>0}\lambda_ie_i+\sum_{i:\nabla_ig(x)<0}(-\lambda_i)(-e_i)+\sum_{\substack{i:\nabla_ig(x)=0\\\lambda_i\ge 0}}\lambda_ie_i+\sum_{\substack{i:\nabla_ig(x)=0\\\lambda_i< 0}}|\lambda_i|(-e_i).
        \end{align*}
        We therefore have
       \begin{align*}
           \tilde{p}(t)  = \sum_{i:\nabla_ig(x)>0}(\lambda_i+t\nabla_ig(x))e_i+\sum_{i:\nabla_ig(x)<0}(-\lambda_i+t|\nabla_ig(x)|)(-e_i)\\+\sum_{\substack{i:\nabla_ig(x)=0\\\lambda_i\ge 0}}\lambda_ie_i+\sum_{\substack{i:\nabla_ig(x)=0\\\lambda_i< 0}}|\lambda_i|(-e_i).
       \end{align*}
    We can now observe that, for $t$ sufficiently large, $\tilde{p}(t)$ is a positive linear combination of vectors in $B$. More specifically, it is a positive linear combination of vectors in $B$ such that $\nabla g(x)^Tb\ge 0$. In other words, letting $B_\text{out}=\{b\in B\mid \nabla g(x)^Tb\ge 0\}$, for $t$ sufficiently large there exist $\beta_1,\ldots,\beta_{|B_\text{out}|}\ge 0$ such that 
    $\tilde{p}(t) = \sum_{b_i\in B_\text{out}}\beta_ib_i$.
    We shall now note that the projection onto the linear subspace $P_{H_C(x)}$ is a linear operation; we therefore get
    \begin{align*}
        P_{T_C(x)}(\tilde{p}(t))&=P_{H_C(x)}(\tilde{p}(t)) = P_{H_C(x)}(p+t\nabla g(x)) \\&= P_{H_C(x)}(p)+P_{H_C(x)}(t\nabla g(x)) \\&= P_{H_C(x)}(p)=p,
    \end{align*}
    and, at the same time, 
    \begin{align*}
        P_{T_C(x)}(\tilde{p}(t))&=P_{H_C(x)}(\tilde{p}(t)) = P_{H_C(x)}\left(\sum_{b_i\in B_\text{out}}\beta_ib_i\right) \\&= \sum_{b_i\in B_\text{out}}\beta_i P_{H_C(x)}(b_i) \\&= \sum_{b_i\in B_\text{out}}\beta_i P_{T_C(x)}(b_i),    
    \end{align*}
    which concludes the proof since vectors $P_{T_C(x)}(b_i)$ are in $D$.
    \end{itemize}
\end{proof}

\begin{figure}[htbp]
		\centering
        \includegraphics[width=0.6\textwidth]{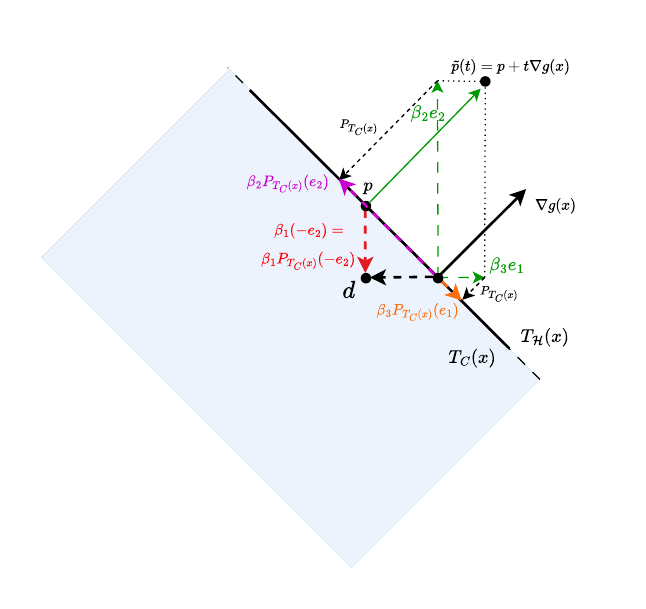}
		\caption{Visualization of the construction of case $x\in\partial C$ of the proof of Proposition \ref{prop:veloc_canon}.}
        \label{fig:proof}
\end{figure}

\section{The algorithmic scheme}
\label{sec:alg}
In this section we finally discuss an algorithmic approach exploiting the analysis made in Section \ref{sec:curves}. Here, we want to tackle problem \eqref{Eq:P} without accessing first-order information on the objective function.
However, it is useful underlining first a derivative-based counterpart of the derivative-free method. 

\subsection{First-order method}
\label{sec:alg_1o}
An algorithmic scheme could be devised with the following update rules:
\begin{align}
    \label{eq:gamma_gradient}
    &\gamma_k\text{ feasible search path s.t.\ }\gamma_k'(0) = P_{T_C(x^k)}(-\nabla f(x^k))=d_k,\\
    &\alpha_k = \max_{j=0,1,\ldots}\{\delta^j\Delta_0\mid f(\gamma_k(\delta^j\Delta_0))\le f(\gamma_k(0))-\sigma (\delta^j\Delta_0)^2\},\\
    \label{eq:deriv_alg}
    &x^{k+1} = \gamma_k(\alpha_k),
\end{align}
where $\delta\in(0,1)$ and $\Delta_0>0$. 
Algorithm \eqref{eq:gamma_gradient}-\eqref{eq:deriv_alg} enjoys global convergence guarantees under a further, reasonable continuity assumption on the feasible search paths $\gamma_k$, as stated in the hereafter.
\begin{assumption}
\label{ass:cont_fp}
    Let $\{x^k\}\subseteq C$ and let $\{\gamma_k\}$ be a sequence such that $\gamma_k$ is a feasible search path at $x^k$ such that $\gamma_k'(0)=P_{T_C(x^k)}(v_k)$ for some sequence of directions $\{v_k\}$. If there exists a subsequence $K\subseteq \{0,1,\ldots\}$ such that $\lim_{k\in K, k\to\infty}x^k = \bar{x}$ and $\lim_{k\in K, k\to\infty}v^k = \bar{v}$, then there exists $\bar \gamma:\mathbb{R}_+\to C$ such that $\bar{\gamma}(t) = \lim_{k\in K,k\to\infty}\gamma_k(t)$ for all $t\in\mathbb{R}_+$ and $\bar \gamma'(0) = P_{T_C(\bar{x})}(\bar{v})$.
\end{assumption}
\begin{proposition}
\label{prop:conv_fo}
    Assume $C$ satisfies assumptions (a1)-(a3). Let $\{x^k\}$ be the sequence generated according to \eqref{eq:gamma_gradient}-\eqref{eq:deriv_alg}, assuming that $x^0\in C$ and that the sequence of feasible search paths $\{\gamma_k\}$ satisfies Assumption \ref{ass:cont_fp}. Then the sequence $\{x^k\}$ admits accumulation points, each one being stationary for the problem.
\end{proposition}
The proof of this result can be found, for the sake of completeness, in Appendix \ref{sec:app}.

As already outlined in Section \ref{sec:curves}, curves $\gamma_k$ defined as $\gamma_k(t) = P_C(x^k-t\nabla f(x^k))$ are always feasible search paths at iterates $x^k$. We shall note here that, in addition, this choice of the curves satisfies Assumption \ref{ass:cont_fp}. Indeed, if $x^k\to \bar{x}$, we have by the continuity of $P_C$ and $\nabla f$ that 
$$\lim_{k\to\infty}P_C(x^k-t\nabla f(x^k)) = P_C(\bar{x}-t\nabla f(\bar{x})) = \bar{\gamma}(t),$$
and $\bar \gamma'(0) = P_{T_C(\bar{x})}(-\nabla f(\bar{x}))$. We can then recover the convergence properties of the projected gradient method with the line search conducted along the projection arc.

\subsection{Derivative-free method}
\label{sec:alg_df}
In this section we finally describe the main contribution of this paper: a derivative-free algorithm for solving problem \eqref{Eq:P} exploiting a fixed set of reference polling directions, i.e., the set $B$ of coordinate directions. The algorithm is formally described in Algorithm \ref{Alg:DF}.


\begin{algorithm}[H]\caption{Feasible search paths method}
	\label{Alg:DF}
	\begin{algorithmic}[1]
		\REQUIRE 
			$\delta \in (0,1)$, $B = \{e_1,\ldots,e_n,-e_1,\ldots,-e_n\}$,  $\sigma>0$, \rev{$\tau>1$, $\bar{\alpha}>0$}.
		\STATE Set $k = 0$ and $\tilde{\alpha}_0 = 1$. Choose $x^0 \in C$.
		\WHILE{stopping condition not satisfied}
        \STATE set success = \textit{False}
        \FORALL{$b_i\in B$}
            \STATE Choose $\gamma_{ik}$ feasible search path at $x^k$ s.t.\ $\gamma_{ik}'(0) = P_{T_C(x^k)}(b_i)$ 
            \IF{$f(\gamma_{ik}(\tilde{\alpha}_k)) \le f(x^k) -\sigma \tilde{\alpha}_k^2 $}
            \STATE $\alpha_k =\tilde{\alpha}_k$
            \STATE set $x^{k+1} = \gamma_{ik}({\alpha}_k)$
            \STATE \rev{$\tilde \alpha_{k+1} = \max\{\bar{\alpha},\tau\alpha_k\}$}
            \STATE set success = \textit{True}
            \STATE \textbf{break}
            \ENDIF
            \ENDFOR
            \IF{success = \textit{False}}
            \STATE set $\alpha_k=0$ and $x^{k+1} = x^k$
            \STATE set $\tilde{\alpha}_{k+1} = \delta \tilde{\alpha}_k$
            \ENDIF
		\STATE set $k=k+1$
        \ENDWHILE
		\RETURN $x^k$
	\end{algorithmic}
\end{algorithm}

Briefly, the method at each iteration scans through a set of feasible search paths, polling points for a given tentative stepsize $\tilde{\alpha}_k$.  As soon as one of these polling points is found providing a sufficient decrease as considered in Proposition \ref{Prop:DecObj}, it is chosen as a new iterate. The tentative stepsize for the next iteration \rev{is then increased, making sure it at least jumps back to a predefined threshold when it is small}. In case none of the considered feasible search paths provides sufficient decrease for the tentative stepsize, the iteration is declared unsuccessful and the tentative stepsize is reduced for the next iteration.

We begin the formal analysis of Algorithm \ref{Alg:DF} studying the asymptotic behavior of the sequences of tentative and actual stepsizes.

\begin{proposition}
\label{prop:stepsizes}
    Let $K\subseteq\{0,1,\ldots\}$ be the sequence of iterations of success in Algorithm \ref{Alg:DF} and let $\bar{K}$ its complementary, i.e., the sequence of unsuccessful iterations.  Let $\{\alpha_k\}$ and $\{\tilde \alpha_{k}\}$ be the sequences of stepsizes produced by Algorithm \ref{Alg:DF}. The following properties hold:
    \begin{enumerate}[(i)]
        \item $\lim\limits_{k\to\infty}\alpha_k=0;$
        \item if $K$ is infinite, then $$\lim_{k\in K,k\to\infty}\tilde\alpha_k=0;$$
        \item $\bar{K}$ is infinite and there exists a subsequence $K_1\subseteq \bar{K}$ such that
        $$\lim_{k\in K_1,k\to\infty}\tilde\alpha_k=0.$$
    \end{enumerate}
\end{proposition}
\begin{proof}
    By the instructions of Algorithm \ref{Alg:DF}, $\{f(x^k)\}$ is nonincreasing. Hence, $\{f(x^k)\}$ admits limit $f^*$ (which is finite by assumption).
    For all iterations $k$, we have
    $$f(x^{k+1})\le f(x^k)-\sigma\alpha_k^2.$$
    Note that in unsuccessful iterations it holds trivially as $\alpha_k=0$ and $x^{k+1}=x^k$. Therefore, taking the limits for $k\to\infty$,
    we have $$f^*\le f^*-\sigma \lim_{k\to\infty}\alpha_k^2,$$
    which implies that $\alpha_k\to0$.

    Now, let us consider the sequence $\{\tilde \alpha_k\}$. We start analyzing the subsequence $K$ of successful iterations. We can immediately note that, since $\alpha_k=\tilde{\alpha}_k$ for all $k\in K$ and $\alpha_k\to0$, if $K$ is infinite, then we have
    \begin{equation}
        \label{eq:proof_steps}
        \lim_{k\to\infty,k\in K}\tilde \alpha_{k}=0.
    \end{equation}

    Now, let us assume that $\bar{K}$ is finite. This would imply that $k\in K$ for all $k$ sufficiently large and thus \rev{$\alpha_k = $ $\tilde{\alpha}_k\ge \bar{\alpha}$}  for all $k$. This would contradict the fact that $\alpha_k\to0$.

    Hence, $\bar{K}$ is infinite. 
    If $K$ is finite, let $\bar{k}$ be the largest index in $K$. Then  \rev{for all $k>\bar{k}$ we have
    $\tilde{\alpha}_k = \tilde{\alpha}_{\bar{k}+1} \delta^{k-\bar{k}-1}$,}  which clearly goes to zero for $k\to\infty$. In this case we would thus have $\lim_{k\in\bar{K},k\to\infty}\tilde{\alpha}_k = 0.$

    Let us now assume that both $K$ and $\bar{K}$ are infinite. Since \eqref{eq:proof_steps} holds, for $k\in K$ sufficiently large we necessarily have \rev{$\tilde{\alpha}_k<\bar{\alpha}$} and thus $(k-1)\in \bar{K}$.
    The sequence $K_1\subseteq \bar{K}$ such that $(k+1)\in K$ for all $k\in K_1$ is thus infinite. We also have $\tilde{\alpha}_k = \delta^{-1}\tilde{\alpha}_{k+1}=\delta^{-1}\alpha_{k+1}$.
    Since $\alpha_k\to 0$, we then get that $\tilde{\alpha}_k\to 0$ for $k\in K_1$, $k\to \infty$.
\end{proof}

We also need to state another result concerning stepsizes and the sufficient decrease condition before turning to the convergence analysis.


\begin{lemma}
\label{lemma:steps_df}
    \rev
    {There exists $\hat{\alpha}\ge \bar{\alpha}$ such that, for any unsuccessful iteration $k$ in Algorithm \ref{Alg:DF} with $k$ sufficiently large, the condition
    $$f(\gamma_{ik}(\hat{\alpha}\delta^j))-f(\gamma_{ik}(0))> -\sigma \hat{\alpha}^2\delta^{2j}$$
    holds for all $i=1,\ldots,2n$ and for all $j$ such that $\tilde{\alpha}_k\le\hat\alpha\delta^j\le \hat{\alpha}$.}
\end{lemma}
\begin{proof}
    Let $m_k$ be the largest index such that $m_k<k$ and iteration $m_k$ is of success. We then know, by the instructions of Algorithm \ref{Alg:DF}, that \rev{$\tilde{\alpha}_k = \tilde{\alpha}_{m_k+1}\delta^{k-m_k-1}$}  and that for all $j=1,\ldots,k-m_k$ iteration $m_k+j$ was of unsuccess. Thus, \rev{ for all $j=1,\ldots,k-m_k$, we have $x^{m_k+j} = x^k$ (which implies $\gamma_{i(m_k+j)}=\gamma_{ik}$), $\tilde{\alpha}_{m_k+j}=\delta^{j-1}\tilde{\alpha}_{m_k+1}$ 
    and, for all $i=1,\ldots,2n$
    $$f(\gamma_{i(m_k+j)}(\tilde{\alpha}_{m_k+1}\delta^{j-1}))-f(\gamma_{i(m_k+j)}(0))> -\sigma \tilde{\alpha}_{m_k+1}^2 \delta^{2(j-1)},$$
    i.e.,
    $$f(\gamma_{ik}(\tilde{\alpha}_{m_k+1}\delta^{j-1}))-f(\gamma_{ik}(0))> -\sigma \tilde{\alpha}_{m_k+1}^2 \delta^{2(j-1)}.$$
    Since the inequality above holds for all $i=1,\ldots,2n$ and for all $j=1,\ldots, k-m_k$, i.e., for all $\tilde{\alpha}_{m_k+1}\delta^h$ such that $\tilde{\alpha}_k=\tilde{\alpha}_{m_k+1}\delta^{k-m_k-1}\le\tilde{\alpha}_{m_k+1}\delta^h\le\tilde{\alpha}_{m_k+1}\delta^0=\tilde{\alpha}_{m_k+1}$, the proof is complete if we can show that $\tilde{\alpha}_{m_k+1}$ is a constant greater or equal than $\bar{\alpha}$ for $k$ sufficiently large.

    Now, we first consider the case where the sequence $K$ of successful iterations is infinite. By Proposition \ref{prop:stepsizes} we know $\tilde{\alpha}_k\to 0$ for $k\in K$, which implies that for $k\in K$ sufficiently large $\tilde{\alpha}_k<\frac{1}{\tau}\bar{\alpha}$ and $\tilde{\alpha}_{k+1} = \bar{\alpha}$. Therefore, since $m_k\to \infty$ and $m_k\in K$, we can deduce that for $k$ sufficiently large $\tilde{\alpha}_{m_k+1} = \bar{\alpha}$. 

    On the other hand, let $K$ be finite and let $\bar{k}$ the last successful iteration. For $k$ sufficiently large we will then have $m_k = \bar{k}$ and thus $\tilde{\alpha}_{m_k+1} = \tilde{\alpha}_{\bar{k}+1}\ge \bar{\alpha}$ by the instructions of the algorithm. 
    }
\end{proof}

We are now able to state the main convergence result of this work.
\begin{proposition}
\label{prop:conv_df}
    Assume $C$ satisfies assumptions (a1)-(a3). Let $\{x^k\}$ be the sequence generated by Algorithm \ref{Alg:DF}, assuming the sequences of feasible search paths $\{\gamma_{ik}\}$ satisfies Assumption \ref{ass:cont_fp}. Then the sequence $\{x^k\}$ produced by Algorithm \ref{Alg:DF} admits accumulation points and there exists at least one accumulation point which is stationary for the problem.
\end{proposition}
\begin{proof}
    Let us consider the sequence of unsuccessful iterations $\bar{K}$, which is infinite according to Proposition \ref{prop:stepsizes}. Let us consider the further subsequence $K_1\subseteq \bar{K}$ such that $\tilde{\alpha}_k\to0$ for $k\in K_1$, $k\to\infty$, which exists by Proposition \ref{prop:stepsizes}.
    
    By the instructions of the algorithm, $\{x^k\}\subseteq C$. Since $C$ is compact, $\{x^k\}$ admits accumulation points. 
    We can apply the same reasoning to $\{x^k\}_{K_1}\subseteq C$ to state that there is at least an accumulation point of  $\{x^k\}_{K_1}$.
    
    Now, let us assume that any of the accumulation points of $\{x^k\}_{K_1}$ are nonstationary; in particular, let $K_2\subseteq K_1$ such that 
    $$\lim_{k\in K_2,k\to\infty}x^k= \bar{x},$$
    with $\bar{x}$ nonstationary, i.e., letting $\bar{d}=P_{T_C(\bar{x})}(-\nabla f(\bar{x}))$, $\|\bar{d}\|\ge\epsilon > 0$ and thus $\nabla f(\bar{x})^T\bar{d}\le - \epsilon^2$.

    Since the iterations in $K_2$ are unsuccessful, we have for all $k\in K_2$ and for all $\gamma_{ik}$, $i=1,\ldots,2n$, that
    $$f(\gamma_{ik}(\tilde{\alpha}_k))-f(\gamma_{ik}(0))> -\sigma \tilde{\alpha}_k^2.$$

    Let now $q$ be an arbitrary positive integer \rev{ and let $\hat{\alpha}$ be the value satisfying Lemma \ref{lemma:steps_df} }. Since $\tilde{\alpha}_k\to0$ for $k\in K_2$, $k\to\infty$, we have $\tilde{\alpha}_k<\rev{\hat{\alpha}}\delta^q$ for $k\in K_2$ sufficiently large and then, by Lemma \ref{lemma:steps_df}, 
    $$f(\gamma_{ik}(\rev{\hat{\alpha}}\delta^q))-f(\gamma_{ik}(0))> -\sigma \rev{\hat{\alpha}^2}\delta^{2q}.$$
    Dividing both sides of the inequality by $\rev{\hat{\alpha}}\delta^q$ and taking the limit for $k\in K_2$, $k\to\infty$, we get 
    $$\lim_{k\in K_2,k\to\infty}\frac{f(\gamma_{ik}(\rev{\hat{\alpha}}\delta^q))-f(\gamma_{ik}(0))}{\rev{\hat{\alpha}}\delta^q}\ge -\sigma \rev{\hat{\alpha}}\delta^q.$$
    Recalling Assumption \ref{ass:cont_fp} for all sequences $\{\gamma_{ik}\}$, we can then write
    $$\frac{f(\bar{\gamma}_{i}(\rev{\hat{\alpha}}\delta^q))-f(\bar{\gamma}_{i}(0))}{\rev{\hat{\alpha}}\delta^q}\ge -\sigma\rev{\hat{\alpha}} \delta^q,$$
    where $\bar{\gamma}_i$ is a feasible search path at $\bar{x}$ such that $\bar{\gamma}_i'(0) = P_{T_C(\bar{x})}(b_i)$.

    Since $q$ is arbitrary in $\mathbb{N}$, we can take the limits for $q\to\infty$, i.e., for $\rev{\hat{\alpha}}\delta^q\to 0$, obtaining 
    $$\mathcal{D}_{f\circ \bar{\gamma}_i}(0) = \nabla f(\bar{x})^T\bar{\gamma}_i'(0)\ge 0.$$
    By Proposition \ref{prop:veloc_canon}, we have $\operatorname{cone}(\{\bar{\gamma}'_i(0)\mid i=1,\ldots,2n\})=T_{C}(\bar{x})$. Since $\bar{d}\in T_C(\bar{x})$, there exist $\beta_1,\ldots,\beta_{2n}\ge0$ such that $\bar{d} = \sum_{i=1}^{2n}\beta_i\bar{\gamma}_i'(0)$; thus we can conclude that
    $$\nabla f(\bar{x})^T\bar{d} = \sum_{i=1}^{2n}\beta_i\nabla f(\bar{x})^T\bar{\gamma}_i'(0)\ge 0,$$
    which is absurd as we had assumed $\nabla f(\bar{x})^T\bar{d}\le -\epsilon^2<0$. Therefore, all limit points of $\{x^k\}_{K_2}$ are stationary.
\end{proof}

\begin{remark}
\label{rem::rem1}
    Once again, as already outlined in Section \ref{sec:curves}, any curve of the form $\gamma_{ik}(t) = P_C(x^k+tb_i)$ is a feasible search path at $x^k$. We also have that, similarly as in the gradient based case, this particular choice of the curves satisfies Assumption \ref{ass:cont_fp}, being the limit curve, by the continuity of $P_C$, equal to $\bar{\gamma}_i(t) = P_C(\bar{x}+tb_i)$ with $\bar{\gamma}_i'(0) = P_{T_C(\bar{x})}(b_i)$. 
    Thanks to this property, we are finally guaranteed that the projection based method soundly achieves the desired result: a stationarity point will be reached without the need of constructing a specific set of search directions for each iterate, based on the local form of the feasible set,\rev{\label{rev:in_remark} nor by asymptotically spanning a dense of search directions} \rev{(which is required, for instance, for \cite{ortho09,Custódio03052024,galvan2021parameter})}; instead, we just need to rely on a suitable fixed set of directions and the projection operator. An example of the resulting curves is shown in Figure \ref{fig:curves}. 
\end{remark}

\begin{figure}[!h]
    \centering
    \includegraphics[width=0.5\linewidth]{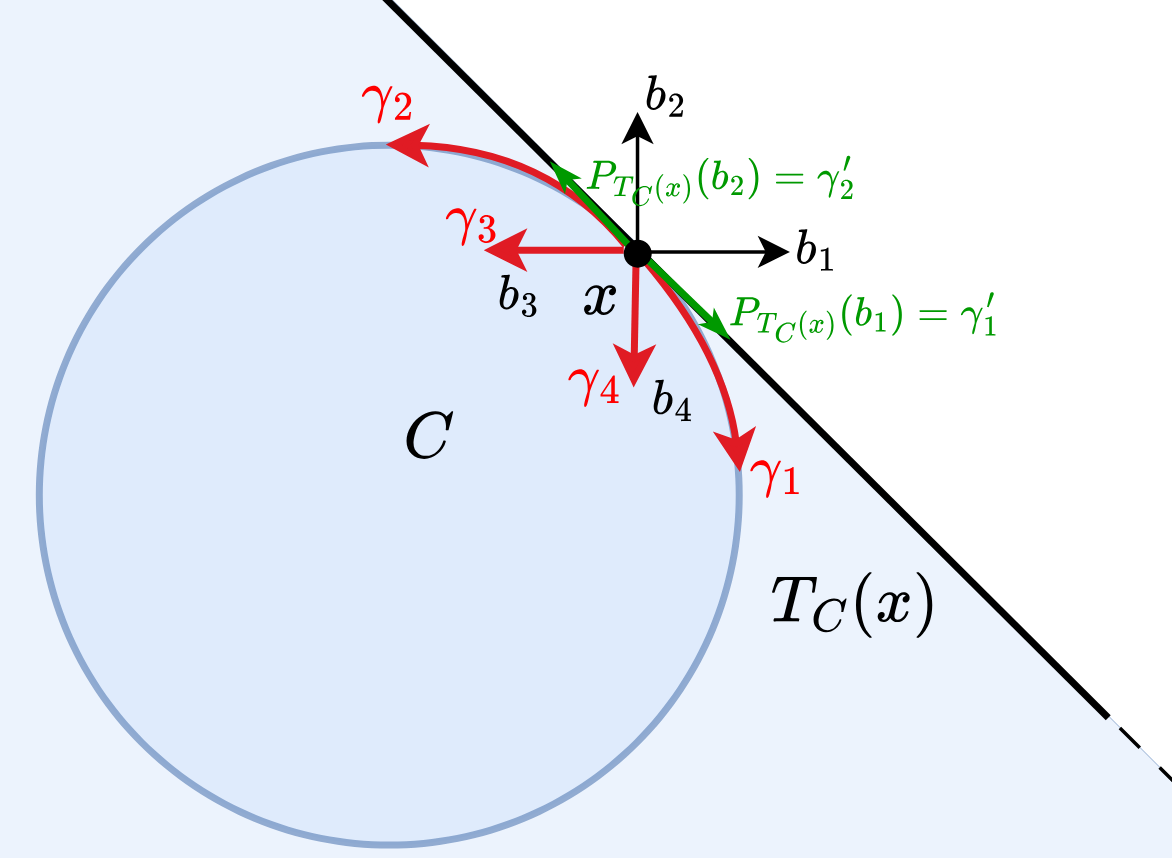}
    \caption{Search curves implicitly defined by the coordinate directions and projection.}
    \label{fig:curves}
\end{figure}

\section{Numerical results}
\label{sec::num_res}

In this section, we show some \rev{numerical} results assessing the effectiveness and efficiency of our approach on a set of standard benchmark problems. The code\footnote{The implementation code of our approach can be found at \url{https://github.com/pierlumanzu/FSP}.} for the experimentation was written in \texttt{Python3}. The experiments were run on a computer with the following characteristics: Ubuntu 24.04 OS, Intel(R) Core(TM) i5-10600KF 6 cores
4.10GHz, 32 GB RAM. 

In the rest of the section, we denote Algorithm \ref{Alg:DF} as \texttt{FSP}. In this experimental setting, \texttt{FSP} was compared with \rev{the following methods:} 
\begin{itemize}
    \item \rev{a state-of-the-art projection-based solver, proposed in \cite{galvan2021parameter}, briefly discussed in Section \ref{sec::intro} and denoted in the following as \texttt{PPM}\footnote{The implementation code of \texttt{PPM} can be found at \url{https://github.com/jth3galv/dfppm}.};}
    \item \rev{the approach introduced in \cite{Custódio03052024}, which integrates the Directional Direct Search (\texttt{DDS}) framework with the Spectral Projected Gradient (\texttt{SPG}) method by replacing the true gradient with a simplex-based approximation; this algorithm is denoted as \texttt{SPG-DDS}; in particular, we tested the same variants considered in the numerical experiments of \cite{Custódio03052024}: \texttt{SPG-DDS-proj}, where infeasible poll points are projected before evaluation, and \texttt{SPG-DDS-barrier}, where infeasible points are discarded;}
    \item \rev{\texttt{DDS-proj}, analogous to \texttt{SPG-DDS-proj} but without the \texttt{SPG} component, which primarily differs from our method by using a dense of search directions in the poll step;}
    \item \rev{\texttt{ORTHOMADS} \cite{ortho09}, a specific instance of the Mesh Adaptive Direct Search (\texttt{MADS}) family of algorithms, which employs at each iteration a deterministic set of $2n$ orthogonal directions.}
\end{itemize}
Regarding the stopping criteria, \rev{all} algorithms had a budget of 10000 function evaluations \rev{unless explicitly stated otherwise}. For \texttt{FSP}, we set $\delta = 0.5$, $\sigma = 10^{-3}$, \rev{$\tau=1.025$ and ${\bar{\alpha}=10^{-6}}$}
. \rev{\label{rev2::opportunistic}Based on preliminary experiments, we adopted a dynamic ordering of the directions in $B$: each iteration was stopped as soon as a direction proved successful, and the next iteration began with that direction. Moreover, the directions $[1,\ldots,1]^\top$ and $[-1,\ldots,-1]^\top$ were included in $B$, and all directions were tested during the first iteration of the method. N}ote that \rev{these rules do} not spoil convergence theory\rev{.} For \texttt{PPM}, we set the parameters values according to the most robust configuration proposed in \cite{galvan2021parameter}, including a dynamic weighting of the distance of the point $x$ to the feasible set $C$. \rev{Both \texttt{FSP} and \texttt{PPM} were stopped when the tentative step sizes fell below a threshold of $10^{-7}$.} \rev{For \texttt{SPG-DDS} variants and \texttt{DDS-proj} we adopted the parameter settings suggested in \cite{Custódio03052024}, using the Matlab code provided by the authors;  for \texttt{ORTHOMADS}, we followed the configuration described in \cite{ortho09} and implemented in the \texttt{NOMAD} library \cite{nomad22}.}

\begin{table}[htb]
    \centering
    \footnotesize
    \renewcommand{\arraystretch}{1.25}
    \begin{tabular}{|c|c||c|}
    \hline
         \textit{Sources} & $n$ & \textit{Problems} \\
         \hline
         \hline
         \multirow{4}{*}{\cite{abramson2008pattern, hock1980test, schittkowski2012more, gould2015cutest}}&$2$& HS22, HS232 \\
         \cline{2-3}
         &$3$& HS29, HS65 \\
         \cline{2-3}
         &$4$& HS43 \\
         \cline{2-3}
         &$6, 7, 8$&AS6, AS7\\
         \hline
         \hline
         \multirow{12}{*}{\cite{gould2015cutest}}&\multirow{3}{*}{$2$}& AKIVA, BEALE, BOXBODLS, BRANIN, \\
         &&BRKMCC, BROWNBS, CAMEL6, CLIFF, \\
         &&CLUSTERLS, CUBE, DANIWOODLS, BOX2 \\
         \cline{2-3}
         &$3$& BARD, YFITU, ALLINIT, BIGGS3 \\
         \cline{2-3}
         &\multirow{2}{*}{$4$}& DEVGLA1, HATFLDB, HIMMELBF, LEVYMONT7, \\
         &&PALMER2, PALMER5D, POWERSUMB \\
         \cline{2-3}
         &$5$&DEVGLA2B, HS45, LEVYMONT8, BIGGS5\\
         \cline{2-3}
         &$6$&HART6, LANCZOS1LS \\
         \cline{2-3}
         &$8$&GAUSS1LS \\
         \cline{2-3}
         &$10$&HILBERTB, TRIGON2 \\
         \cline{2-3}
         &$22$&VANDANMSLS \\
         \cline{2-3}
         &$25$&HATFLDC\\
         \hline
    \end{tabular}
    \caption{Benchmark of problems used for experimentation.}
    \label{tab::problems}
\end{table}

As for the benchmark problems, they are listed in Table \ref{tab::problems}. The first set of instances were also used in \cite{galvan2021parameter} for testing the performance of \texttt{PPM}, while the second set is composed by \texttt{CUTEst} \cite{gould2015cutest} problems with size $n$ ranging from 2 to 25. Unless otherwise stated, all the instances were tested using \rev{a unit hyper-sphere constraint, i.e., $C = \{x \in \mathbb{R}^n\mid\|x-c\|^2 \le 1\}$, where $c \in \mathbb{R}^n$ denotes the center of the hyper-sphere.} We recall that the projection onto such set can be calculated in closed form according to: 
\rev{\begin{equation*}
    P_C(x) = 
    \begin{cases} 
        c+\frac{x-c}{\|x-c\|}, & \mbox{if }x \not\in C, \\ x, & \mbox{otherwise.}
    \end{cases}
\end{equation*}}
Both algorithms start, in each test problem, with the same feasible solution.

To assess the performance of the algorithms, we primarily focused on the number of function evaluations $n_f$ required by each method, i.e., the typical efficiency metric in black-box optimization; we are also interested in the number $n_p$ of times the projection of an unfeasible point is computed; analyzing this quantity provides an insight about efficiency in cases where projection is available, but costly; we also looked the objective function value $f^\star$ returned by each method.
For a compact visualization of the results, we made use of the performance profiles \cite{dolan2002benchmarking} \rev{and the data profiles \cite{dataprof09}, which represent cumulative distribution functions showing, for each solver, the percentage of problems solved within a tolerance $\varepsilon > 0$ given a specific budget of function evaluations}.

\subsection{First set of problems}

In this section, we discuss the results reported in Table{s \ref{tab::first_set_c0}-\ref{tab::first_set_c5}}, obtained by \rev{\texttt{FSP} and \texttt{PPM}} on the first set of problems of Table \ref{tab::problems} \rev{with hyper-sphere centers $c = [0,\ldots,0]^\top$ and $c = [5,\ldots,5]^\top$, respectively}. 

\begin{table}[htb]
   \centering
   \footnotesize
   \renewcommand{\arraystretch}{1.25} 
     \begin{tabular}{|c||c|c|c||c|c|c|} 
     \hline 
     \multirow{2}{*}{\textit{Problem}} &\multicolumn{3}{c||}{\texttt{FSP}} & \multicolumn{3}{c|}{\texttt{PPM}}\\
       \cline{2-7}
       & $f^\star$ & $n_f$ & $n_p$ & $f^\star$ & $n_f$ & $n_p$  \\
       \hline
       HS22& \textbf{1.528} & \textbf{241} & \textbf{128} & \textbf{1.528} & 327 & 256 \\
       \hline
       HS232& \textbf{-0.038} & \textbf{206} & \textbf{109} & \textbf{-0.038} & 434 & 321 \\
       \hline
       HS29& \textbf{-0.192} & \textbf{193} & \textbf{97} & \textbf{-0.192} & 365 & 261 \\
       \hline
       HS65& \textbf{26.548} & \textbf{440} & \textbf{246} & \textbf{26.548} & 553 & 469 \\
       \hline
       HS43& \textbf{-21.435} & 665 & \textbf{365} & \textbf{-21.435} & \textbf{519} & 444 \\
       \hline
       AS6 ($n=6$)& \textbf{2.101} & \textbf{351} & \textbf{177} & \textbf{2.101} & 891 & 722 \\
       \hline
       AS6 ($n=7$)& \textbf{2.708} & \textbf{402} & \textbf{203} & \textbf{2.708} & 1314 & 1088 \\
       \hline
       AS6 ($n=8$)& \textbf{3.343} & \textbf{451} & \textbf{227} & \textbf{3.343} & 3754 & 3424 \\
       \hline
       AS7 ($n=6$)& \textbf{0.0} & 1047 & 26 & \textbf{0.0} & \textbf{313} & \textbf{4} \\
       \hline
       AS7 ($n=7$)& \textbf{0.0} & 1336 & 31 & \textbf{0.0} & \textbf{364} & \textbf{4} \\
       \hline
       AS7 ($n=8$)& \textbf{0.0} & 1628 & 38 & \textbf{0.0} & \textbf{415} & \textbf{4} \\
       \hline
  \end{tabular}
    \caption{Results on the first set of problems listed in Table \ref{tab::problems} \rev{with hyper-sphere centered at $c = [0,\ldots,0]^\top$.}}
    \label{tab::first_set_c0}
\end{table}

\begin{table}[htb]
    \centering
   \footnotesize
   \renewcommand{\arraystretch}{1.25}
   \begin{tabular}{|c||c|c|c||c|c|c|} 
     \hline 
     \multirow{2}{*}{\textit{Problem}} &\multicolumn{3}{c||}{\texttt{FSP}} & \multicolumn{3}{c|}{\texttt{PPM}}\\
       \cline{2-7}
       & $f^\star$ & $n_f$ & $n_p$ & $f^\star$ & $n_f$ & $n_p$  \\
       \hline
       HS22& \textbf{16.0} & \textbf{242} & \textbf{129} & \textbf{16.0} & 336 & 259 \\
       \hline
       HS232& \textbf{-29.373} & \textbf{234} & \textbf{133} & \textbf{-29.373} & 2037 & 2034 \\
       \hline
       HS29& \textbf{-173.494} & \textbf{202} & \textbf{102} & \textbf{-173.494} & 831 & 822 \\
       \hline
       HS65& \textbf{0.0} & 438 & 16 & \textbf{0.0} & \textbf{336} & \textbf{5} \\
       \hline
       HS43& \textbf{-12.436} & \textbf{539} & \textbf{303} & \textbf{-12.436} & 571 & 483 \\
       \hline
       AS6 ($n=6$)& \textbf{77.404} & \textbf{337} & \textbf{176} & \textbf{77.404} & 1822 & 1639 \\
       \hline
       AS6 ($n=7$)& \textbf{91.834} & \textbf{385} & \textbf{201} & \textbf{91.834} & 6883 & 6780 \\
       \hline
       AS6 ($n=8$)& \textbf{106.373} & \textbf{433} & \textbf{226} & \textbf{106.373} & 1727 & 1529 \\
       \hline
       AS7 ($n=6$)& \textbf{126.505} & \textbf{337} & \textbf{176} & \textbf{126.505} & 1023 & 823 \\
       \hline
       AS7 ($n=7$)& \textbf{149.542} & \textbf{385} & \textbf{201} & \textbf{149.542} & 830 & 623 \\
       \hline
       AS7 ($n=8$)& \textbf{172.716} & \textbf{433} & \textbf{226} & \textbf{172.716} & 733 & 488 \\
       \hline
  \end{tabular}
    \caption{\rev{Results on the first set of problems listed in Table \ref{tab::problems} with hyper-sphere centered at $c = [5,\ldots,5]^\top$.}}
    \label{tab::first_set_c5}
\end{table}

We immediately observe that the two algorithms always returned solutions with the same function value. Moreover, an overall advantage of \texttt{FSP} w.r.t.\ its competitor is attested in the table in terms of both $n_f$ and $n_p$. These results are particularly encouraging if we take into account that \texttt{PPM} is based on CS-DFN \cite{galvan2021parameter}, which implements some sophisticated mechanisms aimed at speeding up the computation (such as direction-specific stepsizes, extrapolation steps), compared to the very basic setup of \texttt{FSP}. 
By a careful look at the results, we also get some insight from the \rev{cases of AS7 instances with $c=[0,\ldots,0]^T$ and HS65 instance with $c = [5,\ldots,5]^T$}, where \texttt{PPM} has a better performance: in these problems the solution lies in the interior of the unit hyper-sphere. We can speculate that the refined mechanism of CS-DFN have a greater impact in this ``unconstrained'' scenario; yet, \texttt{FSP} is still able to solve the problem with a reasonable cost; a future integration of the CS-DFN mechanisms within \texttt{FSP} might thus be worth of future investigations. 

\begin{table}[htb]
    \centering
    \footnotesize
    \renewcommand{\arraystretch}{1.25}
    \begin{tabular}{|c||c|c|c|c||c|c|c|c|} 
   	\hline 
   	\multirow{2}{*}{\textit{Problem}} &\multicolumn{4}{c||}{\texttt{FSP}} & \multicolumn{4}{c|}{\texttt{PPM}}\\
   	\cline{2-9}
   	& $f^\star$ & $n_f$ & $n_p$ & $T$ & $f^\star$ & $n_f$ & $n_p$ & $T$  \\
   	\hline
   	\hline
       HS29& \textbf{-22.627} & \textbf{406} & \textbf{221} & \textbf{2.77} & \textbf{-22.627} & 634 & 565 & 4.92 \\
       \hline
  \end{tabular}
    \caption{Results on the original HS29 problem.}
    \label{tab::HS29}
\end{table}

We conclude the section comparing the two approaches on the original formulation of the HS29 problem, whose ellipsoidal constraint respects the feasible set assumptions made in Section \ref{sec::feasible_set}. In this case, the projection onto the feasible convex set $C$ is not available in closed form: the use of a solver is required (we employed the \texttt{CVXOPT}\footnote{\url{https://cvxopt.org}} software package) and the projection operation effectively becomes a significant element of the computational cost of the algorithm. The results are reported in Table \ref{tab::HS29}. We get similar numbers as for HS29 with the unit sphere constraint: both the algorithms reached the same optimal value and \texttt{FSP} outperformed \texttt{PPM} on all the other metrics. In particular, we can observe that \texttt{FSP} performed about 1/2 of the projections of \texttt{PPM}, and such result is also reflected on the overall elapsed time ($T$) spent by the two algorithms.

\subsection{Cutest problems}

In this second experimental section, we \rev{first} compared \texttt{FSP} and \texttt{PPM} on the CUTEst problems reported in Table \ref{tab::problems} \rev{under both types of unit hyper-sphere constraints, i.e., $c = [0,\ldots,0]^\top$ and $c = [5,\ldots,5]^T$}. For a compact view of the results, we report in Figure \ref{fig::pp} the performance profiles achieved by the two approaches in terms of $n_f$ and $n_p$. We shall remark that both the algorithms managed to achieve the same optimal value on all the tested instances; we do not report the numbers here for the sake of brevity.

\begin{figure}
    \centering
    \subfloat{\includegraphics[width=0.49\textwidth]{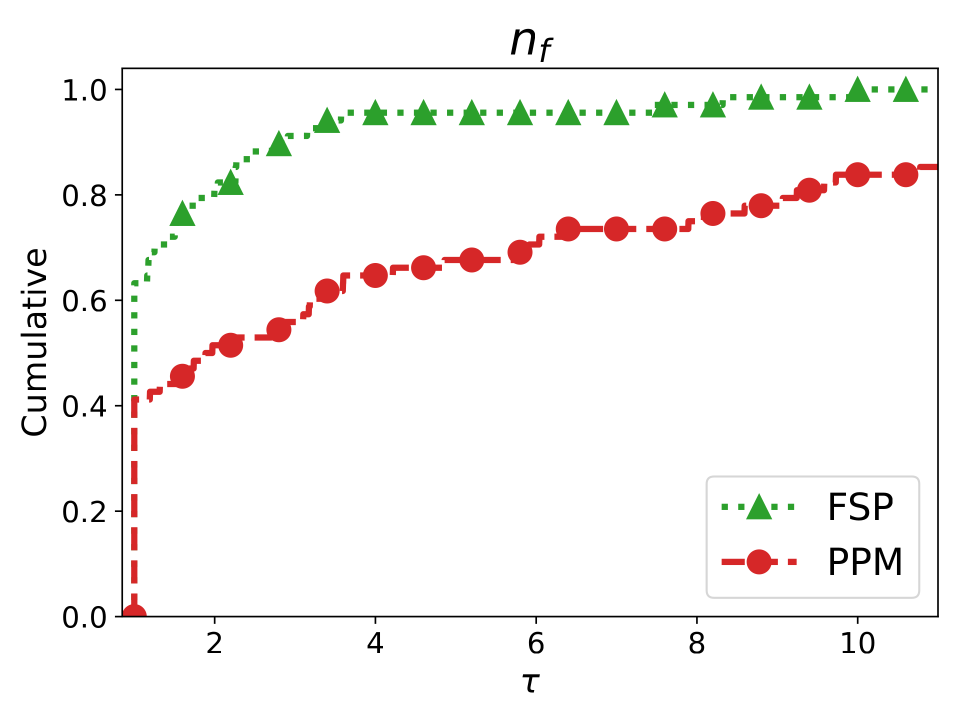}}
    \hfill
    \subfloat{\includegraphics[width=0.49\textwidth]{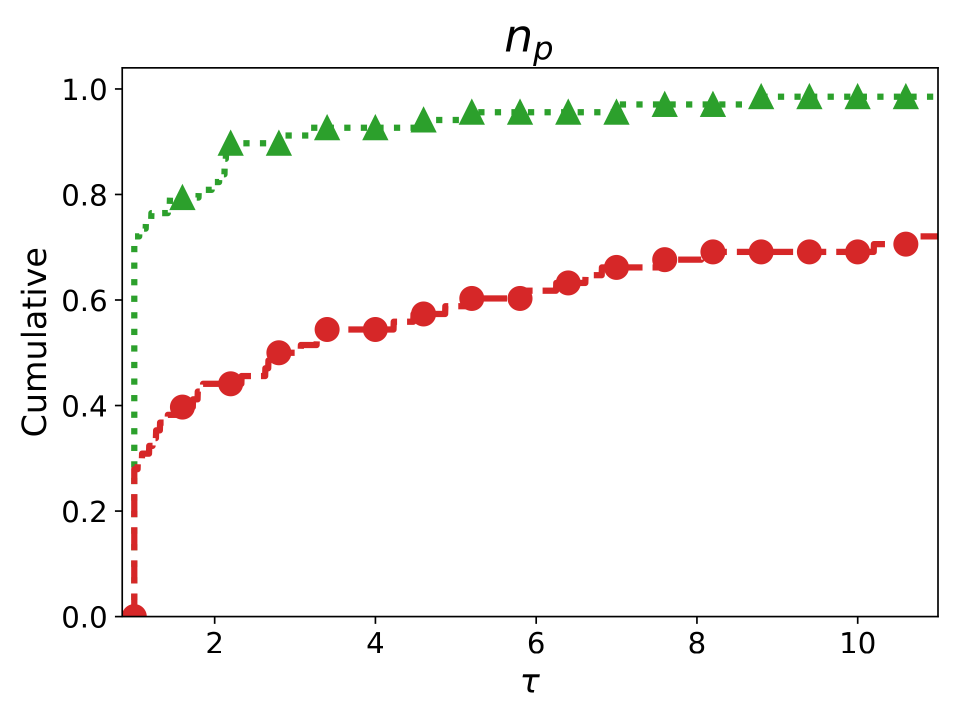}}
    \caption{Performance profiles in terms of $n_f$ and $n_p$ obtained by \texttt{FSP} and \texttt{PPM} on the Cutest problems listed in Table \ref{tab::problems} \rev{under both types of unit hyper-sphere constraints, i.e., $c = [0,\ldots,0]^T$ and $c = [5,\ldots,5]^T$}.}
    \label{fig::pp}
\end{figure}

Again, \texttt{FSP} proved to consume less function evaluations overall w.r.t.\ the competitor. Such performance gap is even more accentuated if we look at the number of  projections $n_p$.

\rev{Finally, we compared our method against the \texttt{SPG-DDS} variants, \texttt{DDS-proj} and \texttt{ORTHOMADS} on the benchmark set, using a budget of 25000 function evaluations. Given the structural differences between the solvers, we utilized data profiles for a concise comparison. Figure \ref{fig::data-SPG} shows these profiles for tolerances $\varepsilon \in \{10^{-2}, 10^{-4}, 10^{-6}, 10^{-8}\}$ against the \texttt{SPG-DDS} variants and \texttt{DDS-proj}, while Figure \ref{fig::data-ORTHO} shows the same profiles against \texttt{ORTHOMADS}.}

\rev{We first observe that \texttt{SPG-DDS-barrier} performed significantly worse than our method, \texttt{SPG-DDS-proj} and \texttt{DDS-proj}. This provides empirical evidence that projection of infeasible points before polling is crucial for obtaining an effective procedure. For accuracy levels $\varepsilon \in \{10^{-2}, 10^{-4}, 10^{-6}\}$, our method was competitive with \texttt{SPG-DDS-proj} and \texttt{DDS-proj}. This arguably indicates that both the use of a dense of search directions and the eventual use of simplex-gradient steps yields no clear advantage over the simpler strategy used within our method. 
When increasing the required accuracy to $\varepsilon = 10^{-8}$, our method even proved to be the most effective choice, outperforming all three competitors.}

\rev{In the comparison with \texttt{ORTHOMADS}, both methods exhibited similar performance at the lowest accuracy level ($\varepsilon = 10^{-2}$). However, as the accuracy requirement increased, our method achieved substantially better results. \texttt{ORTHOMADS} employs a dense of search directions as well, with an extreme-barrier strategy for handling constraints; this combination did not allow it to match the performance of our method, which relies only on coordinate directions.}

\begin{figure}[!h]
    \centering
    \subfloat{\includegraphics[width=0.49\linewidth]{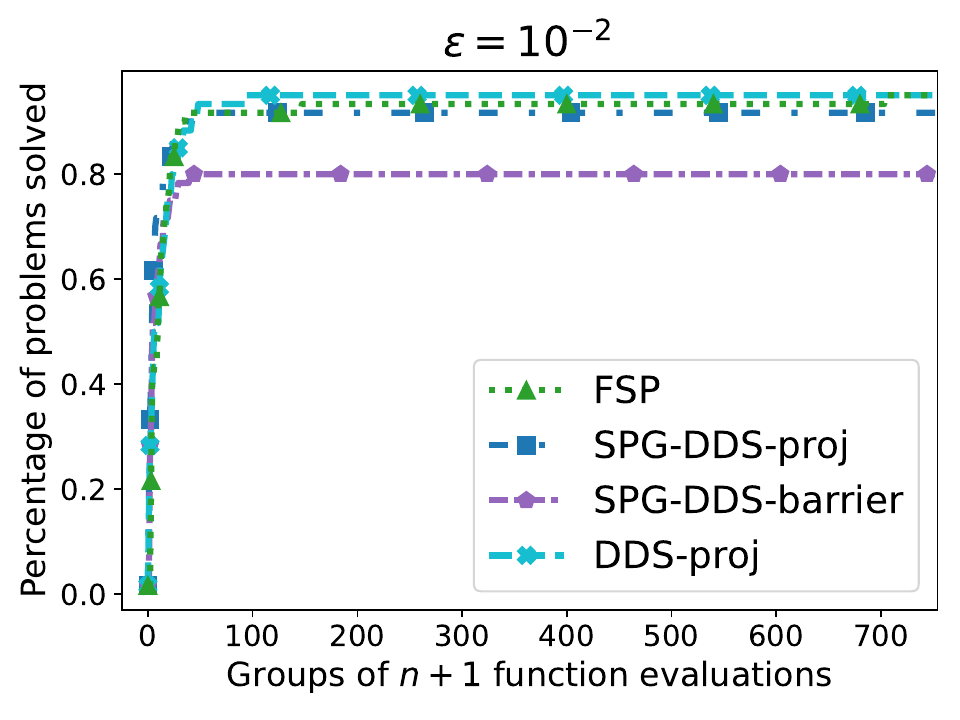}}
    \hfill
    \subfloat{\includegraphics[width=0.49\textwidth]{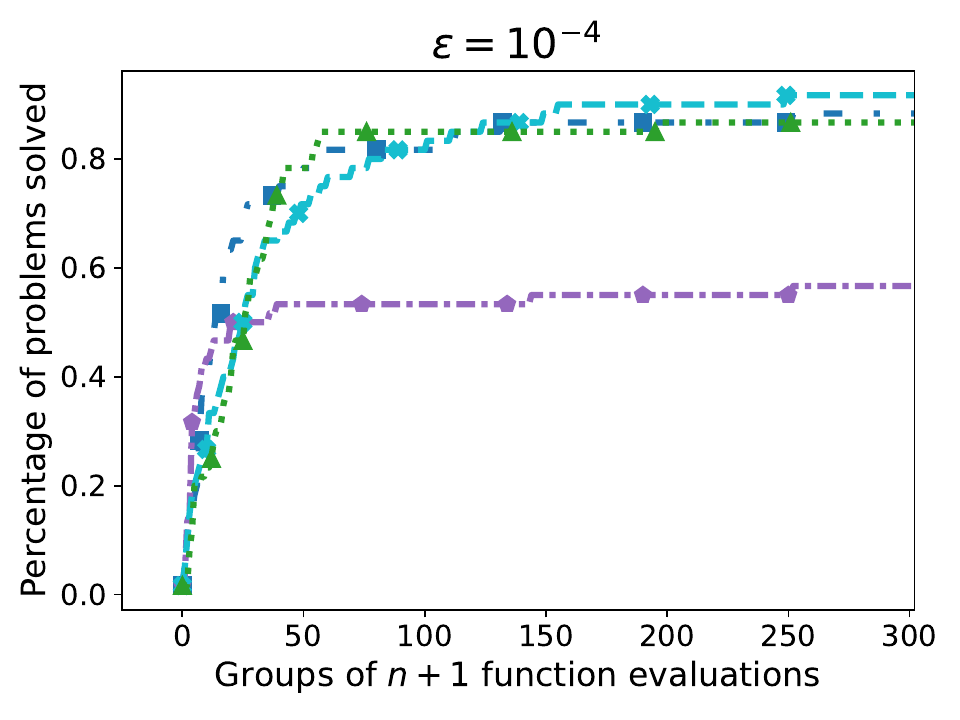}}
    \\
    \subfloat{\includegraphics[width=0.49\linewidth]{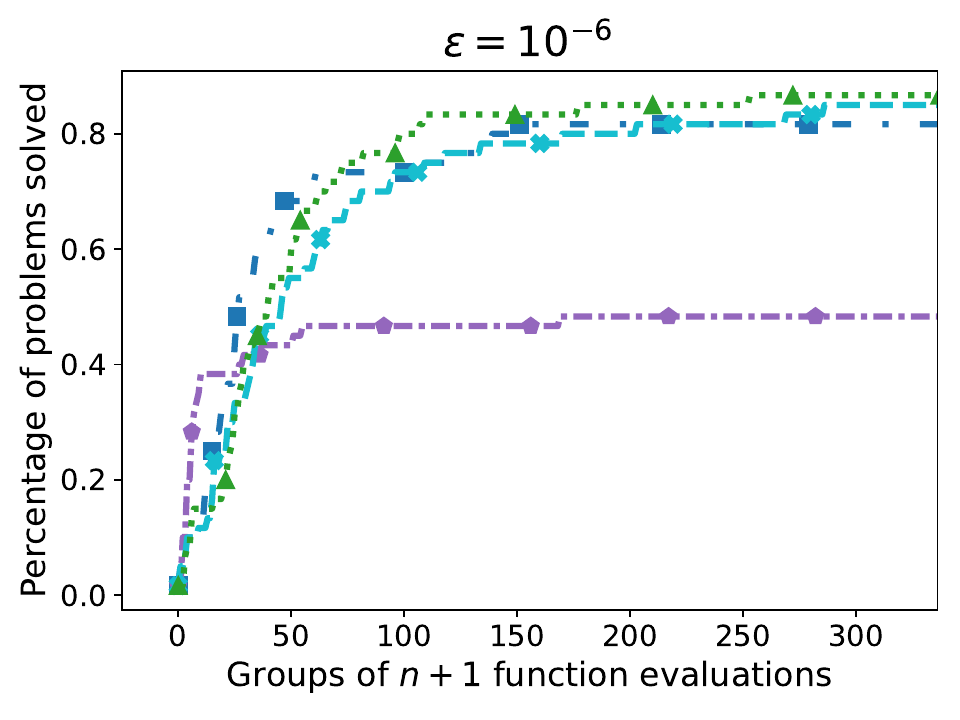}}
    \hfill
    \subfloat{\includegraphics[width=0.49\textwidth]{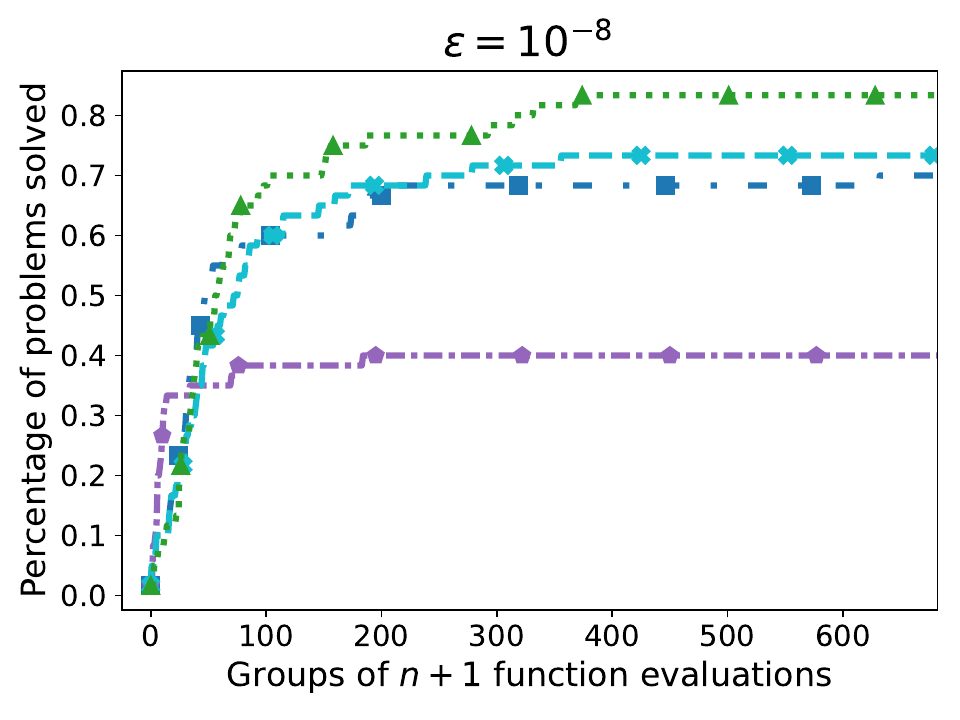}}
    \caption{\rev{Data profiles obtained by \texttt{FSP}, \texttt{SPG-DDS-proj}, \texttt{SPG-DDS-barrier} and \texttt{DDS-proj} for tolerances $\varepsilon \in \{10^{-2}, 10^{-4}, 10^{-6}, 10^{-8}\}$ on the Cutest problems listed in Table \ref{tab::problems}, under unit hyper-sphere constraints with both  $c = [0,\ldots,0]^T$ and $c = [5,\ldots,5]^T$.}}
    \label{fig::data-SPG}
\end{figure}

\begin{figure}[!h]
    \centering
    \subfloat{\includegraphics[width=0.49\linewidth]{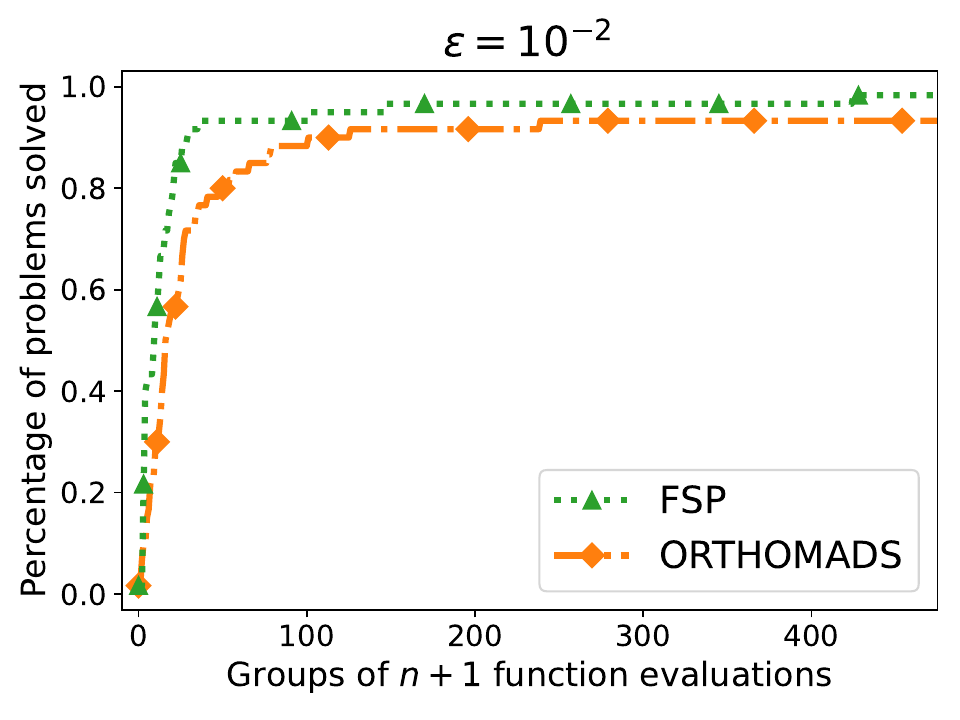}}
    \hfill
    \subfloat{\includegraphics[width=0.49\textwidth]{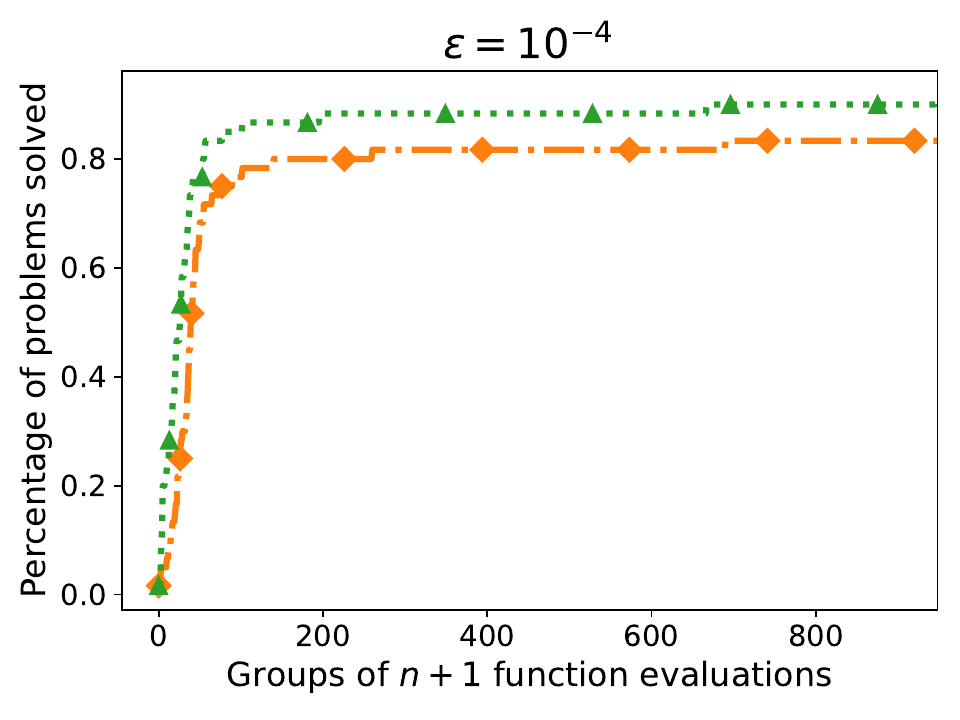}}
    \\
    \subfloat{\includegraphics[width=0.49\linewidth]{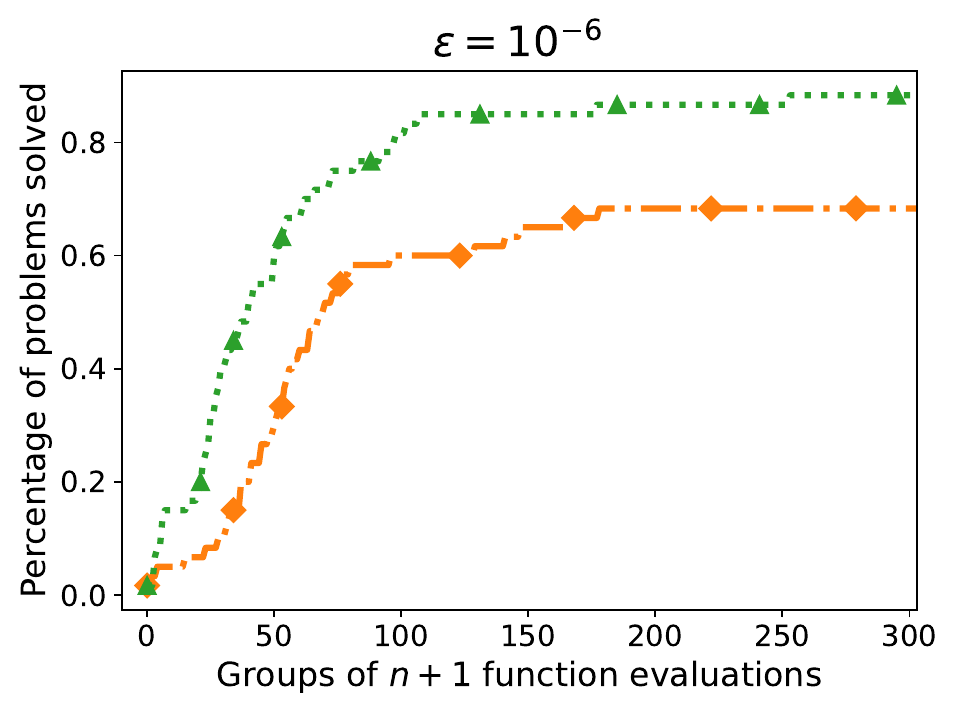}}
    \hfill
    \subfloat{\includegraphics[width=0.49\textwidth]{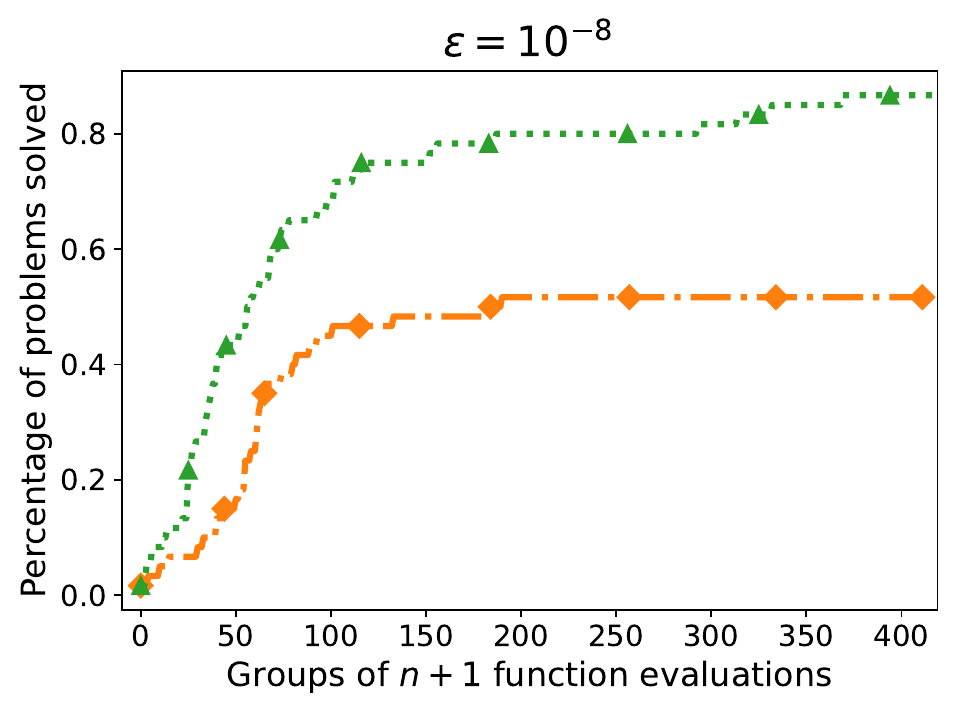}}
    \caption{\rev{Data profiles obtained by \texttt{FSP} and \texttt{ORTHOMADS} for tolerances $\varepsilon \in \{10^{-2}, 10^{-4}, 10^{-6}, 10^{-8}\}$ on the Cutest problems listed in Table \ref{tab::problems}, under unit hyper-sphere constraints with both  $c = [0,\ldots,0]^T$ and $c = [5,\ldots,5]^T$.}}
    \label{fig::data-ORTHO}
\end{figure}

\section{Conclusions}
\label{sec:concl}
In this paper, we focused on black-box optimization problems where the objective function is a smooth function with inaccessible derivatives, whereas the feasible set is a smooth, convex closed set. We introduced the concept of feasible search path, i.e., a curve contained in the feasible set, starting at a feasible solution, with suitable regularity properties. We discussed the properties required by sets of such curves not only to be sufficient for the characterization of stationarity, but also for guaranteeing the decrease of the objective function if a backtracking search is conducted from a nonstationary point. Then, we presented a pattern search algorithm that polls points along feasible search paths, showing that it is provably convergent to stationary points\rev{, without resorting to asymptotically dense sets of directions}. Of particular interest is the special case, fitting the general framework, of search paths defined by the projection of steps carried out along coordinate directions. The corresponding algorithm is shown to be computationally competitive for this class of problems with \rev{a selection of state-of-the-art methods from the literature. Worthy of special attention are the comparisons with other projection-based derivative-free approaches: one transforms the original problem into a nonsmooth unconstrained one, the other mixes (without proven convergence guarantees) our mechanism with an ORTHOMADS-type selection of polling directions and a simplex-gradient approach}.
\rev{\label{rev:fine_conc} Future work might focus on the extension of this contribution to more general convex sets of constraints.}

\section*{Declarations}

\subsection*{Funding}
No funding was received for conducting this study.

\subsection*{Disclosure statement}

The authors report there are no competing interests to declare.

\subsection*{Code availability statement}
The implementation code of the approach presented in the paper can be found at \url{https://github.com/pierlumanzu/FSP}.

\subsection*{Notes on contributors}
\textbf{Xiaoxi Jia} received her M.Sc. degree in mathematics (Optimization) from Nanjing Normal University in China in 2020. She then obtained her Ph.D. degree in Optimization in 2023 from the University of Wuerzburg in Germany. She has completed a one-year and one-month postdoctoral research at Saarland University in Germany. Her main research focuses on nonconvex and nonsmooth optimization problems.
\\\\
\textbf{Matteo Lapucci} received his bachelor and master's degree in Computer Engineering at the University of Florence in 2015 and 2018 respectively. He then received in 2022 his PhD degree in Smart Computing jointly from the Universities of Florence, Pisa and Siena. He is currently Assistant Professor at the Department of Information Engineering of the University of Florence. His main research interests include theory and algorithms for sparse, multi-objective and large scale nonlinear optimization.
\\\\
\textbf{Pierluigi Mansueto} received his bachelor and master's degree in Computer Engineering at the University of Florence in 2017 and 2020, respectively. He then obtained his PhD degree in Information Engineering from the University of Florence in 2024. Currently, he is a Postdoctoral Research Fellow at the Department of Information Engineering of the University of Florence. His main research interests are multi-objective optimization and global optimization.

\subsection*{ORCID}

Xiaoxi Jia: 0000-0002-7134-2169\\
Matteo Lapucci: 0000-0002-2488-5486\\
Pierluigi Mansueto: 0000-0002-1394-0937

\section*{Acknowledgements}
The authors are very grateful to Prof.\ M.\ Sciandrone and Dr.\ T.\ Trinci for the fruitful discussions\rev{, and to Prof. A.\ L.\ Custódio for sharing the \texttt{SPG\_DDS} code.  We would like to thank the two anonymous reviewers for their valuable comments, which helped us significantly improve the quality of the manuscript.}

\bibliographystyle{tfs}
\bibliography{bibliography}

\appendix

\section{Proof of Proposition \ref{prop:conv_fo}}
\label{sec:app}

\begin{proof}
    By the instructions \eqref{eq:gamma_gradient}-\eqref{eq:deriv_alg}, $\{x^k\}\subseteq C$. Since $C$ is compact, accumulation points exist. 
    
    Now, let us assume  that there exists a subsequence $K\subseteq\{0,1,\ldots\}$ such that 
    $$\lim_{k\in K,k\to\infty}x^k= \bar{x},$$
    with $\bar{x}$ nonstationary, i.e., letting $\bar{d}=P_{T_C(\bar{x})}(-\nabla f(\bar{x}))$, $\|\bar{d}\|\ge\epsilon > 0$ and thus $\nabla f(\bar{x})^T\bar{d}\le - \epsilon^2$.
    
    By the instructions of the algorithm we also know that the entire sequence $\{f(x^k)\}$ is nonincreasing: it thus admits limit $f^*$, which is finite since $f$ is bounded below. Moreover, by the sufficient decrease condition we have
    $$f(x^{k+1})-f(x^k)=f(\gamma_k(\alpha_k))-f(\gamma_k(0))\le -\sigma\alpha_k^2\le 0.$$
    Taking the limits for $k\in K$, $k\to\infty$, we immediately get that $$\lim_{k\in K,k\to\infty}\alpha_k=0.$$ 
    Hence, for any $q\in\mathbb{N}$, we have for $k\in K$ sufficiently large that the step $\delta^q\Delta_0$ does not satisfy the sufficient decrease condition, i.e.,
    $$f\left(\gamma_k\left(\delta^q\Delta_0\right)\right)- f(\gamma_k(0))> -\sigma(\delta^q\Delta_0)^2.$$
    Dividing both sides of the above inequality by $\delta^q\Delta_0$ and taking the limit for $k\in K$, $k\to\infty$, we then get
    $$\lim_{k\in K,k\to \infty}\frac{f\left(\gamma_k\left(\delta^q\Delta_0\right)\right)- f(\gamma_k(0))}{\delta^q\Delta_0}\ge  -\sigma\delta^q\Delta_0.$$
    We also know that $\gamma'_k(0)=P_{T_C(x^k)}(-\nabla f(x^k))$ for all $k$ and that $\nabla f(x^k)\to\nabla f(\bar{x})$ as $x^k$ goes to $\bar{x}$. Recalling Assumption \ref{ass:cont_fp} we can then write $ $
    $$\frac{f\left(\bar \gamma\left(\delta^q\Delta_0\right)\right)- f(\bar \gamma(0))}{\delta^q\Delta_0}\ge  -\sigma \delta^q\Delta_0.$$
    Since $q$ is arbitrary in $\mathbb{N}$, we can take the limit for $q\to\infty$, i.e., for $\delta^q\Delta_0\to 0$, obtaining
    $$\mathcal{D}_{f\circ\bar{\gamma}}(0)=\nabla f(\bar{\gamma}(0))^T\bar{\gamma}'(0)\ge  0.$$
    Recalling again Assumption \ref{ass:cont_fp}, we have $\bar{\gamma}'(0)=P_{T_C(\bar{x})}(-\nabla f(\bar{x}))=\bar{d}$ and so we finally get
    $$\nabla f(\bar{x})^T\bar{d}\ge 0,$$
    which is absurd as we had assumed $\nabla f(\bar{x})^T\bar{d}\le -\epsilon^2<0$.

\end{proof}

\end{document}